\documentclass[11pt]{article}
\usepackage[utf8]{inputenc}

\usepackage{amsmath,amsthm,amsopn,amstext,amscd,amsfonts,amssymb,mathrsfs, xcolor}

\usepackage[margin=1.1in]{geometry}

\usepackage{caption}
\usepackage{subcaption}
\usepackage{authblk}

\newtheorem{thm}{Theorem}[section]
\newtheorem{lem}[thm]{Lemma}
\newtheorem{prop}[thm]{Proposition}

\newtheorem{dfn}[thm]{Definition}

\newtheorem{rmk}[thm]{Remark}

\numberwithin{figure}{section}

\newcommand{\refT}[1]{Theorem~\ref{#1}}

\newcommand{\refL}[1]{Lemma~\ref{#1}}
\newcommand{\refR}[1]{Remark~\ref{#1}}

\newcommand{\refS}[1]{Section~\ref{#1}}
\newcommand{\refP}[1]{Proposition~\ref{#1}}

\newcommand{\refand}[2]{\ref{#1} and~\ref{#2}}
\newcommand{\refEq}[1]{(\ref{#1})}

\newcommand{\E}[1]{{\mathbf E}\left[#1\right]}

\newcommand{\p}[1]{{\mathbf P}\left(#1\right)}
\newcommand{\I}[1]{{\mathbf 1}_{[#1]}}

\newcommand{\N}{\mathbb{N}}

\newcommand{\ceil}[1]{\lceil #1 \rceil}

\newcommand{\K}{K}

\title{A degree-biased cutting process for random recursive trees}

\author[1]{Laura Eslava}
\author[2]{Sergio I. L\'opez}
\author[1]{Marco L. Ortiz}
\affil[1]{Instituto de Investigaciones en Matem\'aticas Aplicadas y en Sistemas, Universidad Nacional Aut\'onoma de México.}
\affil[2]{Facultad de Ciencias, Universidad Nacional Aut\'onoma de M\'exico.}

\date{}

\begin{document}
	\maketitle
	
	\abstract
	
	We investigate a degree-biased cutting process on random recursive trees, where each vertex is deleted with probability proportional to its degree. We establish the splitting property and derive the explicit distribution of the number of vertices deleted in each cut. This leads to a recursive formula for $\K_n$, the number of cuts needed to erase a random recursive tree with $n$ vertices. Furthermore, we show that $\K_n$ is stochastically dominated by $J_n$, the number of jumps made by a related walk with a barrier. We prove that $J_n$ converges in distribution to a random variable with a spectrally negative stable distribution. Finally, we examine connections between this cutting procedure and a coalescing process on the set of $n$ elements. \\


AMScode: 05C80, 60F05.

	\section{Introduction}\label{sec: Intro}
	
	An increasing tree on $[n]$ is a rooted labeled tree, where labels along a path away from the root are in increasing order. We write $I_n$ for the set of increasing trees of size $n$; namely, with $n$ vertices. A random recursive tree on $[n]$, denoted $T_n$, is a tree chosen uniformly at random from $I_n$. A classical construction of $T_n$ proceeds as follows: starting with vertex $1$ as the root, for each $2\le k\le n$ a vertex is chosen uniformly at random from $\{1,..., k-1\}$, and vertex $k$ is attached to the chosen vertex.
	
	The concept of cutting down random networks was introduced in \cite{Meir1970} for random rooted labeled trees and soon afterward studied for the case of random recursive trees \cite{Meir1974}. In this original procedure, the goal is to isolate the root by sequentially removing a randomly chosen edge and discarding the subtree that does not contain the root. The procedure ends when only the root remains, at which point we say that the \emph{root has been isolated} or that the \emph{tree has been erased}. 
	
	Note that a vertex-removal process is asymptotically equivalent to the edge-removal process   since removing a uniformly chosen edge is equivalent to removing a vertex uniformly chosen from the current vertex set, excluding the root. On the other hand, a central property of cutting processes, known as the \emph{splitting property} is that, conditionally on its size, the remaining tree after each cut is itself distributed as the original random tree. The probabilistic techniques used to analyze the limiting distribution of $X_n$, the number of random cuts required to erase a random recursive tree with $n$ vertices are based on the splitting property and are extended, for example, to study the number of collisions in a beta-coalescent until coagulation \cite{Iksanov2007,Iksanov2008}.
	
	More specifically, Meir and Moon \cite{Meir1974} exploit the splitting property to obtain the asymptotic limit of the expected value of $X_n$ and Panholzer established the law of large numbers for $X_n$\cite{Panholzer2004}. While using different methods (singularity analysis of the generating functions and a probabilistic argument, reviewed in Section \ref{sec: Jumps}), both Drmota et al, and Iksanov and Möhle deduce that the fluctuations of $X_n$, after an appropriate normalization and centering, converge in distribution to a random variable with a spectrally negative stable distribution \cite{Drmota2009,Iksanov2007}; see \eqref{eq: esp Xn} and \eqref{eq: conv Xn} in Section~\ref{sec: Xn} for a detailed exposition of these results. 
	
	The splitting property determines which of the two models is deemed more suitable for analysis; for example, Galton-Watson trees are studied through vertex-removal processes due to the independent reproduction of their vertices; see for example \cite{ABetal11,Janson06}.
	Baur \cite{Baur2014} provides a comprehensive survey of cutting models and related processes for random recursive trees up to 2014. In 2019, Cai et al. introduced a cutting procedure aimed at network resilience. In their $k$-cut model, edges are randomly selected and removed only after receiving $k$ cuts \cite{Cai2019,Cai2019b,Wang2021}.
	
	The main objective of this paper is to introduce cutting procedures for random recursive trees that speed up the deletion of the tree, meaning that fewer cuts are needed to isolate the root. Eslava et. al. \cite{Eslava2025} introduced a targeted cutting process for random recursive trees.
	Conditionally given the degree sequence of $T_n$, this is a deterministic procedure that significantly reduces the number of cuts required to isolate the root. The process begins by listing the vertices in decreasing order of their degrees and then removing vertices in this order. In the context of Galton-Watson trees, Dieuleveut studied a degree-biased vertex-removal process \cite{Dieuleveut2015}. In this process, a vertex is selected with probability proportional to its degree, then all its descendants (including their subtrees) are removed while the  selected vertex remains. 
	
	In this paper, we introduce the  \emph{degree-biased cutting procedure}, an iterative process of \emph{degree-biased vertex-removals} where vertices are removed, with probability proportional to their degree,
	until the root is isolated or the tree is instantly erased. In either case, we say that the \emph{tree has been erased} (see Definitions \refand{dfn: cut}{dfn: process}). We are concerned with the limiting distribution of $K_n$, the number of cuts required to erase a random recursive tree with $n$ vertices. 
	
	This procedure resembles the model studied by Dieuleveut; however, the key difference is that in our model, the selected vertex is removed rather than its descendants so that the removal scheme preserves the splitting property (see Proposition \ref{prop: split}). 
	Although the splitting property is maintained in the degree-biased cutting process, applying the methods from \cite{Iksanov2007,Iksanov2008} to degree-biased cutting processes presents challenges, see Remark \ref{rmk: atom}. In our case, it is only possible to approximate the number of vertices removed using a jump distribution $\zeta$, leading to an stochastic upper bound for $K_n$ whose asymptotic growth is explicitly known. Our main result, \refT{thm: domination}, together with \refT{thm: J} show that $K_n$ grows faster than $X_n$ by at least by a factor of $\ln n$.
	
	An alternative approach to edge removal is the continuous-time process where the edges are deleted in the order determined by independent exponential clocks. This perspective draws a connection to coalescent processes, first observed in \cite{Goldschmidt2005} and explored further for the degree-biased  vertex-removal in Section~\ref{sec: Related}.
	\subsection{Properties of the degree-biased cutting process}\label{sec: Process}
	
	In what follows, it is relevant to make a difference between random recursive trees $T_n$ and their deterministic feasible values; thus, deterministic trees will be denoted with lowercase letters; e.g., $t\in I_n$.  
	We consider edges in increasing trees as being directed towards the root. In particular, for an increasing tree $t$ and vertex $v\in t$, we define the \emph{degree} of $v$, denoted $\deg_t(v)$, as the number of edges directed towards $v$. 
	If $v_1v_2$ is an edge of an increasing tree, we will implicitly assume that $v_1<v_2$. Whenever $v_1>1$, there is a unique outgoing edge from $v_1$ to a vertex (its parent) which we denote $v_0$.
	
	Let $t'$ be a tree with vertex set $\{v_1,\dots, v_m\}$, where $v_1<\dots<v_m$. We denote by $\Phi_m$ the function that relabels each $v_i\in t'$ with label $i$. In particular, if $t\in I_n$ and $t'\subset t$ then $\Phi_m(t')\in I_m$. In what follows, we abuse notation and write $\Phi(\cdot)=\Phi_m(\cdot)$ whenever the size of the tree is clear from the context. Finally, for a vertex $v\in t$, let $t^{(v)}$ denote the subtree rooted at $v$; i.e. $t^{(v)}\subset t$ contains $v$ and all vertices $w$ such that there is a directed path from $w$ to $v$. 
	The following definition of degree-biased cutting has the property that, at each step, the probability of deleting vertex $v$ is proportional to its degree.
	
	\begin{dfn}[Degree-biased cut]\label{dfn: cut}
		Let $t$ be an increasing tree. Choose an edge $v_1v_2$ uniformly at random and \emph{delete} vertex $v_1$ along with $t^{(v_1)}$. More precisely, if $v_1>1$, then remove the edge $v_0v_1$ and keep only the subtree containing the root of $t$; that is, keep $t^{root}:=t\setminus (t^{(v_1)}\cup v_0v_1)$. Otherwise, $v_1$ is the root of $t$, we let $t^{root}:=\emptyset$ and say that the cut \emph{erases the tree instantly}.  
	\end{dfn}
	
	A key property for our results is that the splitting property holds for the degree-biased cutting of random recursive trees.
	
	\begin{prop}[Splitting property]\label{prop: split}
		Let $n\ge 2$ and let $T_n$ be a random recursive tree. Let $T_n^{root}$ be the tree obtained after a degree-biased cut. For $\ell\in[n-2]$, conditionally given that $\vert T_n^{root}\vert=\ell$, $\Phi (T_n^{root})$ has the distribution of a random recursive tree of size $\ell$.
	\end{prop}
	
	Note that the cutting process removes at least both ends of the selected edge. Consequently, the size of each cut satisfies $2\le |T_n^{(v_1)}|\le n$. Moreover, we obtain its distribution by exploiting the fact that $\E{\deg_{T_n}(1)}=H_{n-1}$, where $H_n=\sum_{i=1}^n i^{-1}$ is the $n$-th \emph{harmonic number}.
	
	\begin{prop}[Size of a cut]\label{prop: size}
		Let $n\ge 2$ and let $T_n$ be a random recursive tree. Let $T_n^{(v_1)}$ be the tree defined by a degree-biased cut of $T_n$; where $v_1v_2$ is a uniformly chosen edge of $T_n$. The distribution of $|T_n^{(v_1)}|$ is given by
		
		\begin{align}\label{eq: size}
			\p{|T_n^{(v_1)}|=k}=
			\frac{nH_{k-1}}{(n-1)(k+1)k}\I{2\le k<n} +
			\frac{H_{n-1}}{n-1} \I{k=n}.
		\end{align}
	\end{prop}
	
	We note in passing that the probability that the degree-biased cutting procedure ends in one step has a relatively large mass coming from an \emph{instant deletion} of the tree; see Remark~\ref{rmk: atom}. 
	
	\begin{dfn}[Degree-biased process]\label{dfn: process}
		Let $n\ge 2$ and $T_n$ be a random recursive tree. Sequentially perform degree-biased cuts until the resulting tree (or empty set) cannot be cut anymore; at which point we say that $T_n$ was \emph{erased}. More precisely, let  $(G_t)_{t\ge 0}$ be a sequence of trees such that $G_0=T_n$ and, for $t\ge 1$, if $|G_{t-1}|\ge 2$, then $G_t=G_{t-1}^{root}$; otherwise, $G_t=G_{t-1}$. The number of cuts needed to erase $T_n$ is defined by
		\begin{align*}
			K_n=\min\{t\ge 1: |G_t|\in \{0,1\}\};
		\end{align*}
		the halting cases correspond to either instantly erasing the tree or isolating the root.
	\end{dfn}
	
	\subsection{Main results}\label{sec: Main}
	
	Let $\{\xi_i: i\in\N\}$ be a sequence of independent copies of a random variable $\xi$ 
	taking positive integer values and $n\in \N$. The \emph{random walk with a barrier} at $n$, $\{R^{(n)}_i:i\in\N_0\}$ is defined recursively as $R^{(n)}_0=0$ and, for $i\geq 1$, \[R^{(n)}_i:= R^{(n)}_{i-1}+ \xi_i\mathbf{1}_{\{R_{i-1}^{(n)}+\xi_i< n\}}. \] 
	Throughout this section, consider $\xi=\zeta$ with mass probability, for $k\ge 2$,
	\begin{align}\label{dfn: z}
		\p{\zeta=k}= \frac{H_{k-1}}{k(k+1)}.
	\end{align}
	
	Let $J_n$ denote the number of jumps of the process $\{R^{(n+1)}_i : i\in\N_0\}$ when $\xi=\zeta$; that is, $J_n:= \sum_{i\geq 0} \mathbf{1}_{\{R^{(n+1)}_{i-1}+\xi_{i}\le n\}}$. 
	\refP{prop: split} allows us to obtain a recursive formula for $\K_n$ and to compare $\K_n$ with $J_n$, akin to \cite[Lemma 1]{Iksanov2007}.
	Our main result states that $J_n$ may be interpreted as a stochastic overcount of $K_n$.  	
	\begin{thm}\label{thm: domination}
		$J_n$ stochastically dominates $\K_n$.
	\end{thm}
	
	Unfortunately, Proposition \ref{prop: size} implies that the degree-biased cutting places some extra mass of instant deletion at each step of the process, preventing $J_n$ from being a good approximation of $K_n$; see Remark~\ref{rmk: atom}. Nevertheless, 
	the next result gives us insight into the deletion time $K_n$.    
	Using results in \cite{Iksanov2008}, which are summarized in Section~\ref{sec: Jumps},  we obtain the asymptotic behavior of both the mean and the limiting distribution of $J_n$ under an appropriate normalization.

	\begin{thm}\label{thm: J}
		As $n$ goes to infinity, $\E{J_n}\sim 2n(\ln n)^{-2} $ and, in probability, 
		\[\frac{J_n}{\E{J_n}}\stackrel{p}{\longrightarrow} 1, \]
		and
		\[\frac{J_n}{n}(\ln{n})^3-2\ln{n}  \]
		converges in distribution to a random variable $4X$, where $X$ is a random variable with characteristic function 
		\begin{equation*}
			\E{e^{itX}}=\exp\left(it\ln |t|-\frac{\pi}{2}|t| \right).
		\end{equation*} 
	\end{thm}

	In conclusion, the degree-biased cutting process results in quicker deletion times compared to the random cutting approach, at least by a factor of $\ln n$ in the first order of $X_n$; compare \refT{thm: J} with \eqref{eq: esp Xn} and \eqref{eq: conv Xn}.
	In contrast, the degree-biased procedure is significantly slower when compared to deleting vertices deterministically based on their ordered degrees, as in the targeted vertex-cutting process, where the number of required cuts has asymptotic growth at most $n^{1-\ln 2}$ \cite{Eslava2025}.
	
	Further research into this topic includes obtaining tighter bounds for the deletion time of both the degree-biased cutting process and the targeted vertex-cutting. In addition, it would be interesting to develop a better understanding of the asymptotic properties of the size of the subtrees removed at each step of the degree-biased cutting procedure, as this is key to further understanding the properties of the related coalescent process.
	
	The outline for the paper is the following. 
	In \refS{sec: Jumps} we review the results for random walks with a barrier at $n$ \cite{Iksanov2008} and place the known properties of $X_n$ into such context. 
	
	In \refS{sec: biased cuts} we 
	establish Propositions \ref{prop: split} and \ref{prop: size}; namely, the splitting property and the distribution of the mass lost at each degree-biased cut of the process. To do so, we compute the joint probability of $T^{root}$ and $T^{(v_1)}$, the tree removed after the first degree-biased cut, see \refL{lem: joint}; and use a first-step analysis to recover a recursive formula for $\K_n$, see \refR{prop: recursive}. Section \ref{sec: jump process} contains the proofs of Theorems \ref{thm: J} and \ref{thm: domination}.
	
	Finally, in Section \ref{sec: Related} we briefly explore the coalescent process associated with the degree-biased cutting procedure.

	{\bf Notation.}
	We use $|A|$ to denote the cardinality of a set $A$. For $n\in \N$ we write $[n]:=\{1,2,\dots, n\}$. $H_n=\sum_{i=1}^n i^{-1}$ is the $n$-th \emph{harmonic number}. We denote natural logarithms by $\ln$. For real functions $f,g$ we write $f(x)\sim g(x)$ when $\lim_{x\to\infty}f(x)/g(x)=1$, $f(x)=o(g(x))$ when $\lim_{x\to \infty}  f(x)/g(x) =0$ and $f(x)=O(g(x))$ when $|f(x)/g(x)|\le C$ for some $C>0$. We use $\stackrel{p}{\longrightarrow}$ 
	to denote convergence in probability. 
	
	\section{Jumps for random walks with a barrier}\label{sec: Jumps}
	
	Recall the process $\{R^{(n)}_i:i\in\N_0\}$, defined in \refS{sec: Main}, it follows from  the definition that the process $\{R^{(n)}_i : i\in\N_0\}$ has non-decreasing paths and satisfies $R_i^{(n)} < n$ for all $i \in \N_0$. Let $M_n= |\{i\in\N \vert R^{(n)}_{i-1}\neq R^{(n)}_i\}|= \sum_{i\geq 0} \mathbf{1}_{\{R^{(n)}_{i-1}+\xi_{i}< n\}}$ denote the number of jumps of the process $\{R^{(n)}_i : i\in\N_0\}$. Note that $M_n\le n$ since $\xi\ge 1$ with probability one. 
	
	The asymptotic behavior of $M_n$ has been studied when $\xi$ is distributed as in \eqref{eq: z} in \cite{Iksanov2007} and for broader classes of $\xi$ in \cite{Iksanov2008}; in particular, for the case of an infinite-mean $\xi$. In the latter, the results are stated under the assumption of $\xi$ having a probability atom at $1$; however, a straightforward adaptation allows us to present their results without such restriction. 
	
	\begin{prop}[Theorem 1.1 in \cite{Iksanov2008}]\label{prop: esp Mn}
		If $\sum_{j=1}^n \p{\xi\ge j}\sim \ell(n)$ for some  function $\ell$ slowly varying function at infinity then, as $n\to \infty$, $\E{M_n}\sim n/\ell(n)$ and, in probability,
		\begin{align*}
			\frac{M_n}{\E{M_n}} \stackrel{p}{\longrightarrow} 1. 
		\end{align*}
	\end{prop}
	
	\begin{prop}[Theorem 1.4 in \cite{Iksanov2008}]\label{prop: conv Mn}
		If $\p{\xi\geq n}\sim \frac{\ell(n)}{n}$ for  some function $\ell(n)$ slowly varying at infinity and $\E{\xi}=\infty$, the normalization \[ \frac{M_n-b(n)}{a(n)}\] converges in distribution to a random variable $X$ with characteristic function \[\E{e^{itX}}=\exp\left(it\ln |t|-\frac{\pi}{2}|t| \right)\] whenever the following conditions hold for positive functions $a(x),b(x)$ and $c(x)$: 
		\begin{itemize}
			\item[\it i)] $\lim_{x\to\infty}x\p{\xi \geq c(x)}=1$,
			\item[\it ii)] $b(\phi(x))\sim\phi(b(x))\sim x$ where 
			$$\phi(x):=x\int_{0}^{c(x)} \p{\xi>y}dy;$$
			\item[\it iii)] $a(x)\sim x^{-1}b(x)c(b(x))$.
		\end{itemize}
	\end{prop}
	
	The proofs of Propositions \ref{prop: esp Mn} and \ref{prop: conv Mn} rely on a coupling of the process $R^{(n)}_i$ with the classic random walk $S_i=\sum_{j=1}^i \xi_j$. Observe that the jumps of $S_i$ correspond precisely to the independent variables $\{\xi_i: i\in \N\}$; while the jumps of $R^{(n)}_i$, being conditioned to not exceed the barrier at $n$, lose both the independence and identical distribution property. 
	
	The limiting distribution of $N_n:= \inf\{i\in\N: S_i\geq n \}$ is well understood and it depends on the tails of $\xi$; moreover, $N_n$ may be interpreted as the number of jumps of $S_i$ before exceeding the barrier at $n$. The key idea in \cite{Iksanov2007,Iksanov2008} is to prove that the difference $M_n-N_n$, under a natural coupling of $(R^{(n)}_i, S_i)$, is negligible with respect to the first order of $N_n$.
	
	Other related works on the domain of attraction of $\alpha$-stable distributions are the following. Erickson studies the asymptotic behavior of the spent and residual times associated with a renewal process \cite{Erickson70}. Geluk and Haan derive the theory of stable probability distributions and their domains of attraction via Fourier transforms \cite{Geluk97}.
	Berger surveys the case of random walks in the Cauchy domain of attraction \cite{berger_notes_2019}.
	
	\subsection{Uniform random cuts}\label{sec: Xn}
	
	Consider $T_n$ a random recursive tree on $n$ vertices and perform a random cut; then the number of vertices removed is given by $\xi^{(n)}$, with 
	\begin{equation}\label{eq: jump xi}
		\p{\xi^{(n)}=k}=\frac{n}{(n-1)k(k+1)},
	\end{equation}
	for $k\in [n-1]$; see e.g. \cite{Meir1974}. On the other hand, let $\xi$ be defined, for $k\in \N$, by 
	\begin{equation}\label{eq: z}
		\p{\xi=k}=(k(k+1))^{-1}.
	\end{equation}
	Then $\p{\xi\ge n}=n^{-1}$ and the distribution of $\xi$ conditionally given $\xi<n$, coincides with the distribution of $\xi^{(n)}$. It follows that $M_n$ corresponds precisely to $X_n$, the number of uniform random cuts required to isolate the root of $T_n$. 
	
	The known properties of $X_n$ may be recovered using Propositions~\ref{prop: esp Mn} and \ref{prop: conv Mn}. Namely, as $n$ goes to infinity, $\E{X_n}\sim n(\ln n)^{-1}$, and in probability, 
	\begin{equation}\label{eq: esp Xn}
		\frac{X_n}{\E{X_n}} \stackrel{p}{\longrightarrow}1;
	\end{equation}
	moreover,  
	\begin{align}\label{eq: conv Xn}
		\frac{(\ln n)^2}{n}X_n-\ln n-\ln(\ln n)    
	\end{align}
	converges in distribution to a random variable with a spectrally negative stable distribution.
	
	\section{Analysis of the degree-biased cutting process}\label{sec: biased cuts}
	
	Consider a random recursive tree $T_n$, $n\ge 2$. In what follows, we let $v_1v_2$ be a uniformly chosen edge in $T_n$ so that, after performing a degree-biased cut to $T_n$, we obtain trees $T_n^{(v_1)}$ and $T_n^{root}$. The following proposition computes the joint distribution of their relabeled versions.
	
	\begin{lem}\label{lem: joint}
		Let $T_n$ be a random recursive tree of size $n$ and $T_n^{root}$, $T_n^{(v_1)}$ be the resulting subtrees of a degree-biased cut. For $2\leq k<n$, $t'\in I_{k}$ and $t''\in I_{n-k}$  we have that
		
		\begin{equation}\label{eq: joint}
			\p{\Phi(T_n^{(v_1)})=t', \Phi(T_n^{root})=t''}=\dfrac{n \text{deg}_{t'}(1)}{(n-1)(k+1)!(n-k-1)!};
		\end{equation}
		while, for $k=n$ and $t'\in I_n$,
		\begin{equation}\label{eq: Ts=0}
			\p{\Phi(T_n^{(v_1)})=t'}=\frac{\text{deg}_{t'}(1)}{(n-1)(n-1)!}.
		\end{equation}
	\end{lem}
	
	\begin{proof}
		Note that, conditionally given $v_1v_2$, we have either $v_1=1$, so that $\Phi(T_n^{(v_1)})=T_n$ and $T_n^{root}=\emptyset$, or $v_1>1$ which implies that $T_n=T_n^{root}\cup T_n^{(v_1)}\cup \{v_0v_1\}$ and $2\le |T_n^{(v_1)}|< n$. 
		
		For the first case, let $t,t'\in I_n$, then $$\p{T_n^{(v_1)}=t'|T_n=t}=\p{v_1=1|T_n=t}\I{t'=t}.$$
		The probability on the right side, given that $t=t'$, is precisely the ratio between the degree of the root of $t'$ and the number of edges. Therefore
		\begin{equation*}
			\p{\Phi(T_n^{(v_1)})=t'}=\sum_{t\in I_n}\p{T_n^{(v_1)}=t'\vert T_n=t}\p{T_n=t}=\frac{\text{deg}_{t'}(1)}{n-1}\p{T_n=t'}.
		\end{equation*}
		Furthermore, $\vert I_n\vert=(n-1)!$, so that $\p{T_n=t'}=\frac{1}{(n-1)!}$ and \eqref{eq: Ts=0} is obtained. 
		
		Let $2\leq k<n$, $t'\in I_{k}$ and $t''\in I_{n-k}$. Similarly, it suffices to verify that 
		\begin{align}\label{eq: cond}
			\p{\Phi(T_n^{(v_1)})=t', \Phi(T_n^{root})=t''\vert T_n=t}= \frac{\text{deg}_{t'}(1)}{n-1}\I{t\in C_{t'}};
		\end{align}
		where $C_{t'}$ is the set of increasing trees $t$ of size $n$ such that there is a vertex set $\{v_0,\dots,v_{k}\}$ such that {\it i)} $t^{(v_1)}$ has vertex set $ \{v_1,\dots,v_{k} \}$, {\it ii)} $\Phi(t^{(v_1)})=t'$, {\it iii)} $v_0v_1\in t$ and {\it iv)} $\Phi(t\setminus t')=t''$ (in particular, the vertex set of $t\setminus t'$ is $[n] \setminus \{v_1,\dots,v_{k}\}$).
		
		The above conditions imply that $|C_{t'}|=\binom{n}{k+1}$. To see this, let 
		$L= \{v_0,\dots,v_k\}\subset [n]$ and $W=\{w_1, \ldots, w_{n-k}\}=[n]\setminus \{v_1,\ldots, v_k\}$ satisfy $v_i<v_{i+1}$ and $w_j\le w_{j+1}$ for all $i,j$. Denote by $t'_L$ the tree obtained after relabelling each vertex $i\in t'$ with vertex $v_i$; similarly, let $t''_W$ be the tree obtained after relabelling each vertex $j\in t''$ with vertex $w_j$. It is clear that each $t\in C_{t'}$ has the form $t=t'_L\cup t''_W\cup\{v_0v_1\}$ for some set $L\subset [n]$ of size $k+1$.
		
		To verify \eqref{eq: cond} we observe that the probability vanishes whenever $t\notin C_{t'}$. On the other hand, conditionally given that $t=t'_L\cup t''_W\cup \{v_0v_1\} \in C_{t'}$, the event $\{\Phi(T_n^{(v_1)})=t', \Phi(T_n^{root})=t''\}$ is equivalent to choosing one of the $\text{deg}_{t'}(1)$ edges that are children of $v_1$ in $t$, this probability is uniform over $t\in C_{t'}$.
		Finally, \eqref{eq: cond} together with $\p{T_n=t}=\frac{1}{(n-1)!}$ yields \eqref{eq: joint}.
	\end{proof}
	
	The joint distribution of $(T_n^{(v_1)}, T_n^{root})$ depends only on the size and degree of the tree that is to be cut and not the tree that remains. Thus, conditionally given the size of $|T_n^{root}|=\ell$, we have that $T_n^{root}$ is equally likely on $I_\ell$, so that Proposition \ref{prop: split} is established.
	A similar argument, together with an averaging of $\text{deg}_{t'}(1)$ over $t'\in I_k$, yields Proposition \ref{prop: size}.
	
	\begin{proof}[Proofs of Propositions \refand{prop: split}{prop: size}]
		Let $1< k\le n$ and $\ell=n-k$, we identify $\E{\text{deg}_{T_k}(1)}$ with the following expression  
		\begin{align}\label{eq: Hk}
			\sum_{t'\in I_k} \frac{\text{deg}_{t'}(1)}{(k-1)!}=\E{\text{deg}_{T_k}(1)}= H_{k-1}.
		\end{align}
		We compute the marginal distribution of $\Phi(T_n^{root})$. If $\ell\ge 1$, using \eqref{eq: joint} and \eqref{eq: Hk}, we get for any $t''\in I_{\ell}$,
		\begin{align*}
			\p{\Phi(T_n^{root})=t''}
			=\sum_{t'\in I_{k}} \frac{n \text{deg}_{t'}(1)}{(n-1)(k+1)!(n-k-1)!}
			=\frac{nH_{k-1}}{(n-1)(k+1)(k)(n-k-1)!};
		\end{align*}
		note that the probabilities above depend only on $k=n-|t''|$, establishing Proposition \ref{prop: split}. On the other hand, since $|I_\ell|=(\ell-1)!=(n-k-1)!$ we have 
		\begin{align*}
			\p{|T_n^{root}|=\ell} = \sum_{t''\in I_{\ell}} \p{\Phi(T_n^{root})=t''}
			=\frac{nH_{k-1}}{(n-1)(k+1)k};
		\end{align*}
		this establishes the first term of \eqref{eq: size} since $|T_n^{root}|=\ell$ if and only if $|T_n^{(v_1)}|=k$. Finally, from \eqref{eq: Ts=0} and \eqref{eq: Hk}, we get the case $k=n$, 
		\begin{align*}
			\p{|T_n^{(v_1)}|=n}= \sum_{t'\in I_n} \frac{n\text{deg}_{t'}(1)}{(n-1)(n-1)!}=\frac{nH_{n-1}}{n-1}
		\end{align*}
		completing the proof of Proposition \ref{prop: size}.
	\end{proof}
	
	\begin{prop}\label{prop: recursive}
		Let $K_0=K_1=0$ and $K_n$, $n\ge 2$, be as in Definition~\ref{dfn: process}. Then $\p{K_2=1}=\p{K_3=1}=1$ and for $n\ge 4$, 
		\begin{align}
			\p{\K_n=1} &= \dfrac{H_{n-2}}{(n-1)^2} +\dfrac{H_{n-1}}{n-1}, \label{eq: K=1} \\
			\p{\K_n=j} &=  \sum_{k=2}^{n-2} \dfrac{nH_{k-1}}{k(k+1)(n-1)}\p{\K_{n-k}=j-1};  & j\geq 2. \label{eq: K>1}
		\end{align}
		
	\end{prop}
	
	\begin{proof}
		Let $T_n$ be a random recursive tree and let $T^{(v_1)}_n$ be the tree defined after one degree-biased cut of $T_n$, and $K_n$ be the number of cuts necessary to erase $T_n$.
		
		If the size of $T^{(v_1)}_n$ is either $n-1$ or $n$, then $T_n$ is instantly erased and $K_n=1$. By \eqref{eq: size} in  Proposition~\ref{prop: size}, we obtain \eqref{eq: K=1} for $n\ge 2$; in particular, $K_2=K_3=1$ almost surely.
		
		Let $n\ge 4$ and $j\ge 2$. By the splitting property, Proposition~\ref{prop: split}, we have 
		\[\p{\K_n=j}=\sum_{k=2}^{n-2} \p{\vert T_n^{(v_1)} \vert=k} \p{\K_{n-k}=j-1};\]
		indeed, the cutting procedure after one cut continues on a tree of size $n-k$, conditionally given $|T^{(v_1)}_n|=k$, regardless of the labels in $T_n\setminus T^{(v_1)}_n$. Using \eqref{eq: size} in the expression above, we obtain \eqref{eq: K>1}.    
	\end{proof}
	
	\subsection{Harmonic-biased random walk with a barrier}\label{sec: jump process}
	
	As we mentioned before, when $\xi$ has distribution as in \eqref{eq: z}, the process $(R^{(n)}_i)_{i\ge 0}$ corresponds precisely to the number of jumps of the classic cutting process of random recursive trees. In the case of the degree-biased cutting, conditionally given a tree of size $n$, the size of the cut $\zeta^{(n)}$ follows the distribution
	$\p{\zeta^{(n)}=k}=f_n(k)n/(n-1)$ where 
	\begin{align}\label{dfn: f}
		f_n(k):=\frac{H_{k-1}}{k(k+1)}\I{2\le k<n} + \frac{H_{k-1}}{k}\I{k=n}. 
	\end{align}
	
	Unfortunately, there is no distribution $\zeta$ on positive integers such that, for $2\le k\le n$,
	\begin{align}\label{eq: contradiction}
		\p{\zeta=k|\zeta\le n}=\p{\zeta^{(n)}=k}
	\end{align}
	To see this, write $p_k:=\p{\zeta=k}$. Rearranging the terms of \eqref{eq: contradiction}, we need 
	\begin{align}\label{eq: equiv}
		\frac{n}{n-1}\sum_{i=2}^n p_i =\frac{p_k}{f_n(k)}
	\end{align}
	to hold for all $2\le k\le n$. Since the left side does not depend on $k$ and converges to one, we infer that it is necessary that $p_k=\frac{H_{k-1}}{k(k+1)}$ for all $k\ge 2$. However, in such a case, \eqref{eq: equiv} fails to hold when $k=n$ and $n\ge 3$.
	
	\begin{rmk}\label{rmk: atom}
		The probability that the degree-biased cutting procedure is erased in one step has a relatively large mass coming from an \emph{instant deletion} of the tree (case $k=n$ in \eqref{dfn: f}). Thus, $\zeta^{(n)}$ has anomalous behaviour at its largest atom $n$, preventing an accurate approximation by an unbounded random variable $\zeta$.  
	\end{rmk}
	
	In what follows we let $(R^{(n)}_i)_{i\ge 0}$ be the random walk with a barrier on $n$ with jump distribution $\zeta$ given by $\p{\zeta=k}=\frac{H_{k-1}}{k(k+1)}$ and let $J_n$ be the number of jumps of the process $(R^{(n+1)}_i)_{i\ge 0}$; see Section \ref{sec: Main}. Then $\zeta$ is our best jump distribution approximation for the number of vertices discarded after each degree-biased cut. The role of $\zeta$ is analogous to $\xi$, defined \refEq{eq: jump xi}, for the classic random cutting process of $T_n$.
	
	Let $\tau=\inf\{i\in \N_0: R^{(n+1)}_i>0\}$. By the strong Markov property, the process $(R^{(n+1)}_{\tau+ j})_{j\ge 0}$ is independent of $(R^{(n+1)}_i)_{0\le i\le \tau}$ given $R^{(n+1)}_\tau$. In particular, conditionally given $R^{(n+1)}_\tau=k$, $1<k\le n$, $(R^{(n+1)}_{\tau+ j})_{j\ge 0}$ has the same distribution as $(R^{(n-k+1)}_j)_{j\ge 0}$. 
	Let $J_0=J_1=0$. Similarly to \refP{prop: recursive}, for $j\ge 2$, 
	\begin{align}\label{eq: recursive}
		\p{J_n=j}= \sum_{k=2}^{n-2} \p{\zeta=k \,\vert\, \zeta\leq n} \p{J_{n-k}=j-1}.
	\end{align}
	
	We now prepare for the proofs of Theorems \refand{thm: J}{thm: domination}. Using a telescopic argument, we have for $i\geq 0$,
	\begin{equation}\label{eq: telescopic}
		\sum_{k=i+1}^n\frac{1}{k(k+1)}= \sum_{k=i+1}^n	\frac{1}{k}-	\frac{1}{k+1}=\frac{1}{i+1}-\frac{1}{n+1}.
	\end{equation}
	In turn, for $n\ge 2$,
	\begin{equation}\label{eq: dist xi}
		\p{\zeta\leq n} = \sum_{k=2}^{n}\sum_{i=1}^{k-1}\dfrac{1}{ik(k+1)}=\sum_{i=1}^{n-1}\dfrac{1}{i}\sum_{k=i+1}^{n}\dfrac{1}{k(k+1)} =1- \frac{1}{n} -\frac{H_{n-1}}{n+1}.
	\end{equation} 
	This verifies, in passing, that $\zeta$ is a random variable on $\N$. In addition, note that 
	\begin{align*}
		\left( 1-\frac{1}{n} -\frac{H_{n-1}}{n+1}\right)^{-1} = \frac{n(n+1)}{n^2-nH_{n-1}-1}\ge \frac{n+1}{n-H_{n-1}}\ge \frac{n}{n-1}.
	\end{align*}
	Consequently, for $2\le k<n$, 
	\begin{align}\label{eq: key}
		\p{\zeta=k|\zeta\le n}= \frac{H_{k-1}}{k(k+1)}\left( 1-\frac{1}{n} -\frac{H_{n-1}}{n+1}\right)^{-1}\ge \p{\zeta^{(n)}=k}.    
	\end{align}
	
	\begin{proof}[Proof of \refT{thm: domination}]
		For $n\le 3$, $K_n=J_n$ almost surely. 
		For $n\ge 4$, it suffices to verify that 
		$\p{J_n\geq j}\ge \p{\K_n\geq j}$ for all $j\in \N$. Clearly $\p{J_n\geq 1}=\p{\K_n\geq 1}=1$, so henceforth we assume that $j\ge 2$. 
		From the expression in \eqref{eq: recursive}, by changing the order of the sums, we get
		\begin{align}
			\p{J_n\geq j}&= \sum_{i=j}^{\infty}\sum_{k=2}^{n-2} \p{\zeta=k \,\vert\, \zeta\leq n} \p{J_{n-k}=i-1} \nonumber \\ 
			&=\sum_{k=2}^{n-2} \p{\zeta=k \,\vert\, \zeta\leq n} \p{J_{n-k}\ge j-1}; \label{eq: middle} 
			\intertext{analogously, using \eqref{eq: K>1},}
			\p{\K_n\geq j}&= \sum_{k=2}^{n-2} \p{\zeta^{(n)}=k} \p{\K_{n-k}\geq j-1}. \label{eq: last}
		\end{align}
		The proof of Theorem \ref{thm: domination} is complete by induction on $j\ge 2$ and $n\ge 4$ using \eqref{eq: key}--\eqref{eq: last}.
	\end{proof}
	
	\begin{proof}[Proof of \refT{thm: J}]
		
		From \eqref{eq: dist xi} we have that the tails of $\zeta$ satisfy 
		\begin{equation}\label{eq: tail xi}
			\p{\zeta \geq n}= 1-\p{\zeta \le n-1}= \frac{1}{n-1}+\frac{H_{n-2}}{n}=\frac{\ln n}{n}(1+o(1)).
		\end{equation}
		
		Thus, 
		\[\sum_{j=1}^n \p{\zeta \geq j}= \left(\sum_{j=1}^n \frac{H_j}{j}\right)(1+o(1)) \sim\frac{(\ln n)^2}{n}.\]
		
		Therefore, conditions of \refP{prop: esp Mn} are satisfied with $\ell(x)=(\ln n)^2/2$ for $x> 1$. Hence, $\E{J_n}\sim 2n(\ln n)^{-2}$ and $J_n/\E{J_n}$ converges in probability to one as $n\to\infty$.
		
		We now verify that the conditions of \refP{prop: conv Mn} are satisfied for $a(x)=4x(\ln x)^{-3}$, $b(x)=2x(\ln x)^{-2}$ and $c(x)=x\ln x$.
		
		First, substitution of $n=x\ln x$ in \refEq{eq: tail xi} yields 
		\[ x\p{\zeta\geq x\ln x}= x\p{\zeta\geq\ceil{ x\ln x}}= \frac{x(\ln x+\ln\ln x)}{x\ln x}(1+o(1))\to 1, \text{ as } x\to\infty,\] where $y\mapsto \ceil{y}$ is the ceiling function.
		
		Again, by using \refEq{eq: tail xi} in the definition of $\phi(x)$ and 
		\begin{align}\label{eq: ln2}
			\int_1^n \frac{\ln x}{x} dx = \frac{(\ln n)^2}{2},
		\end{align}
		we have $\phi(x)\sim\frac{x(\ln x)^{2}}{2}$ since
		\[\phi(x)=
		2x+x \int_{2}^{x\ln{x}}\frac{\ln y}{y} (1+o(1)) dy=\frac{x(\ln (x\ln x))^2}{2}(1+o(1)) \sim\frac{x(\ln x)^{2}}{2};\]
		where we used that $\ln (x\ln x)\sim \ln x$.
		Similarly, the fact that $\ln x$ is a slowly varying function implies that $\ln 2x -2\ln\ln x \sim\ln x + (2\ln\ln x-\ln 2)\sim \ln x$. Therefore
		
		\[\phi(b(x)) \sim\frac{x(\ln 2x -2\ln\ln x)^{2}}{(\ln{x})^2}\sim x\sim \frac{x(\ln x)^2}{(\ln x + (2\ln\ln x-\ln 2))^2}\sim b(\phi(x));\]
		whereas,   
		\[x^{-1}b(x)c(b(x)) =\dfrac{4x(\ln{x}-2\ln{\ln{x}}-\ln 2)}{(\ln{x})^4}\sim a(x).\]
		
		Proposition \ref{prop: esp Mn} then implies that 
		\begin{align*}
			\frac{J_n- b(n)}{a(n)}= \frac{J_n}{4n}(\ln n)^3-\frac{\ln n}{2}
		\end{align*}
		converges in distribution to a random variable $X$ with characteristic function \[\E{e^{itX}}=\exp\left(it\ln |t|-\frac{\pi}{2}|t| \right),\]
		which recovers the statement in \refT{thm: J}.
	\end{proof}
	

	\section{A related coalescent process}\label{sec: Related}
	
	There is a natural coalescent process associated with the cutting procedure of a random tree. A coalescent process is a Markov process that takes its values in the set of partitions of $\mathbb{N}$ and it is generally described by its restrictions to the set $[n]$, for $n\in \mathbb{N}$.
	
	The general procedure for obtaining an $n$-coalescent from a cutting procedure is the following. At any given step, we have a tree whose vertex set $\{v_1,v_2,\ldots, v_m\}$ is the blocks of a partition of $[n]$; moreover, the blocks are listed in increasing order of least elements and all paths starting from the root $v_1$ are increasing. Select a random vertex $v$ (which in the cutting process would be removed) and add the labels of the removed vertices into the label $v$; in other words, coalesce all the blocks of the subtree of $v$, inclusive, to form a tree on a new label-set. The $n$-coalescent is the process of the partition of $[n]$ corresponding to the label set; starting from a tree with all singletons.  
	
	Goldschmidt and Martin first proposed the connection between cutting procedures and coalescent processes \cite{Goldschmidt2005}. They identified that the coalescent starting from a random recursive tree $T_n$ defines the Bolthausen Sznitman coalescent on $\mathbb{N}$. Conditionally given that there are $b$ blocks in the partition, for $2\leq k\leq b\leq n$, the rate at which any given set of $k$ labels coalesce is
	\begin{equation}\label{eq: BS rates}
		\frac{(b-k)!(k-2)!}{(n-1)!}.
	\end{equation}
	
	Procedures to obtain beta-coalescent use uniformly random binary trees and Galton-Watson trees with offspring distribution in the domain of attraction of a stable law of index $\alpha\in [1/2,1)$ \cite{Abraham2013}.
	
	A similar process is studied in \cite{pitters2016} for plane-oriented recursive trees (also identified as preferential attachment trees): Select a vertex $v_1$ uniformly at random. If $v_1$ is a leaf, do nothing. Otherwise, select one of its successors, say $v_2$, uniformly at random. Proceed as the process described above with the edge $v_1v_2$. In this case, the resulting coalescent process is related to an arcsine coalescent; however, the finite coalescent processes are not consistent, that is, these cannot be extended to a coalescent process on $\mathbb{N}$.
	
	\subsection{Degree-biased version}
	
	We define the degree-biased coalescent process as follows: start with $T_n$ a random recursive tree on $[n]$ and associate an independent exponential random variable with mean $1$ to each edge. This random variable is the time when the edge is selected to perform the first step of the degree-biased cutting process.
	
	If the edge $\ell_i\ell_j$ is selected and $i>1$, then there exists $\ell_h$ such that $\ell_h\ell_i$ is an edge on $T_n$. The labels in the subtree $T^{(\ell_i)}$ are instantly added to the label of the vertex $\ell_h$. That is, the new set of labels replaces $\ell_h$ with $\ell_h\cup\{s: s\in T_n^{(\ell_i)}\}$.

	The proof of the following proposition is based on \refP{prop: size} and the ideas in \refL{lem: joint}.
	
	\begin{prop}\label{prop: rate}
		The first coalescing event corresponding to the degree-biased cutting of a random recursive tree of size $n$ merges any given set of $k$ labels at rate 
		
		\begin{align}
			\lambda'_{n,k}&= \dfrac{(n-k)!(k-2)!}{(n-1)!}H_{k-2} \quad \text{for $3\leq k<n$}, \label{eq1}\\
			\lambda'_{n,n}&= H_{n-1} +\dfrac{H_{n-2}}{n-1}. \label{eq2}
		\end{align} 
	\end{prop}
	
	\begin{proof}
		We start with the case $k=n$ corresponding to having all singletons merged in one step; from the degree-biased cutting procedure, this is equivalent to having the process end in exactly one step. Namely, either the tree is instantly erased (an edge adjacent to the root is selected in the cutting process) or the root has exactly one vertex attached, vertex 2, and the selected edge $v_1v_2$ has $v_1=2$. In other words, the probability that all 
		blocks coalescent corresponds to the probability that $|T_n^{(v_1)}|\in \{n-1,n\}$. Since the next coalescent event is the minimum of $n-1$ exponential random variables with rate 1, the total rate of instant coagulation is, by \refP{prop: size},
		\begin{align*}
			\lambda'_{n,n}= (n-1) \p{|T_n^{(v_1)}|\in \{n-1,n\}}= H_{n-1}+\frac{H_{n-2}}{n-1}.
		\end{align*}
		For the case $3\le k<n$, let $L=\{v_0,v_1,\ldots, v_{k-1}\}$ be a set of $k$ vertices in $T_n$ such that $v_0<v_1<\cdots <v_{k-1}$. Then, the rate $\lambda_L$ at which elements in $L$ coalesce is given by 
		\begin{align*}
			\lambda_L= (n-1) \p{V(T_n^{(v_1)})=L}.
		\end{align*}
		For each $t'\in I_{k-1}$, let $C_{L,t'}$ be the set of increasing trees $t$ of size $n$ such that {\it i)} $t^{(v_1)}$ has vertex set $ \{v_1,\dots,v_{k-1} \}$, {\it ii)} $\Phi(t^{(v_1)})=t'$ and {\it iii)} $v_0v_1\in t$. By a similar argument as that in the proof of \refL{lem: joint}, we have
		\begin{align*}
			\p{V(T_n^{(v_1)})=L}&= \sum_{t'\in I_{k-1}}\sum_{t\in C_{L,t'}} \p{T_n=t, \Phi(T_n^{(v_1)})=t'} 
			& = \sum_{t'\in I_{k-1}}\sum_{t\in C_{L,t'}} \frac{\text{deg}_{t'}(1)}{(n-1)(n-1)!};
		\end{align*}
		using $|C_{L,t'}|=|I_{n-k+1}|=(n-k)!$ together with \eqref{eq: Hk}, we get 
		\begin{align*}
			\lambda_L=\frac{(n-k)!}{(n-1)!} \sum_{t'\in I_{k-1}} \text{deg}_{t'}(1) =\frac{(n-k)!(k-2)!}{(n-1)!} H_{k-2}.
		\end{align*}
		The rate $\lambda_L$ depends only on the size of $L$; completing the proof of \eqref{eq1}.
	\end{proof}
	
	As we mentioned before, these coalescents can not be extended to a $\Lambda$-coalescent process on $\N$ since their rates do not satisfy the consistency equations $\lambda'_{n,k}=\lambda'_{n+1,k}+\lambda'_{n+1,k+1}$. For example, the following expression does not vanish,
	\begin{align*}
		\lambda'_{n,n}+\lambda'_{n,n-1}-\lambda'_{n-1,n-1}
		&=H_{n-1}-H_{n-2} +\frac{H_{n-2}}{n-1}-\frac{H_{n-3}}{n-2}+\frac{H_{n-3}}{(n-1)(n-2)};
	\end{align*}
	after a careful rearrangement and cancellation of terms, it simplifies to 
	\begin{align*}
		\frac{1}{n-1}  +\frac{H_{n-2}}{n-1} -\frac{H_{n-3}}{n-1}= \frac{1}{n-1}+\frac{1}{(n-1)(n-2)}
		=\frac{1}{n-2}>0.
	\end{align*}
	
	\bibliographystyle{plain}
	
\end{document}